\documentclass[a4paper,12pt]{article}

\usepackage{tikz}
\usepackage{amssymb}
\usetikzlibrary{matrix,arrows,backgrounds}
\usepackage[plainpages=false]{hyperref}
\usepackage{amsfonts,latexsym,rawfonts,amsmath,amssymb,amsthm,mathrsfs}
\usepackage{amsmath,amssymb,amsfonts,latexsym,lscape,rawfonts}

\usepackage[all]{xy}
\usepackage{eufrak}
\usepackage{makeidx}         
\usepackage{graphicx,psfrag}

\usepackage{array,tabularx}

\usepackage{setspace}
\usepackage{geometry}

\newtheorem{thm}{Theorem}[section]

\newtheorem{lem}[thm]{Lemma}
\newtheorem{clm}[thm]{Claim}
\newtheorem{prop}[thm]{Proposition}

\theoremstyle{remark}
\newtheorem{rmk}[thm]{Remark}

\theoremstyle{definition}
\newtheorem{Def}[thm]{Definition}                                        %

\def \C {\mathbb C}

\def \U {\mathcal U}

\title{Remarks on the H\"older-continuity of solutions to parabolic equations with conic singularities}

{\small{\author{Yuanqi Wang}}}

\date{\vspace{-5ex}}

\begin{document}
\maketitle
\begin{abstract} This is  a note on  \cite{LSU} and  \cite{FS}. Using their work line by line, we prove  the H\"older-continuity of solutions to linear parabolic  equations of mixed type, assuming   the coefficient of $\frac{\partial}{\partial t}$ has time-derivative bounded from above.  On a K\"ahler manifold,  this H\"older estimate works when the metrics possess conic singularities along a normal crossing divisor.

\end{abstract}
\section{Introduction}
\indent 

Historically, H\"older-continuity of solutions to linear elliptic and  parabolic equations (in various cases) has been proved and extensively studied  by De Giorgi \cite{DeGiorge}, Nash \cite{Nash}, Moser \cite{Moser}, Krylov-Safonov \cite{KrylovSafonov}. Many other experts have contributed to this topic as well. In addition to the above articles, we also refer the interested readers to \cite{Caffarelli}, \cite{FS}, \cite{GT}, \cite{LSU}, and the references therein.   In \cite{FS}, Safonov-Ferretti give a unified proof of the H\"older-continuity in both the divergence case and non-divergence case. The key is to establish growth properties  for the level sets of the solutions. 
   
    In this note, we focus on the  work in  \cite{FS} on divergence-form equations, and the related work in \cite{LSU}. The operator in equation (\ref{equ considered by FS}) is exactly the one considered in \cite{FS}. The little difference is: $a_{0}$ in \cite{FS} is not allowed to depend on time (line 11 in page 89), but here \textbf{we allow $a_{0}$ to depend on time}. 

   The motivation of us is to study the heat equation associated with a Ricci flow.  The Ricci flow is a special time-parametrized family  of Riemannian metrics $g(t)$. Given a time-family  of Riemannian metrics $g(t)$ over a Euclidean ball $B$, the heat equation of this family reads as 
   \begin{equation}\label{equ heat wrt Riemannian metrics}
 \frac{\partial u}{\partial t}-\Delta_{g}u\triangleq  \frac{\partial u}{\partial t}-\frac{1}{\sqrt{detg_{ij}}}\frac{\partial}{\partial x_{i}}(g^{ij}\sqrt{detg_{ij}}\frac{\partial u}{\partial x_{j}})  =f, 
\end{equation}   
where $x_{i}$'s are the Euclidean coordinates. To estimate  the H\"older norm of $u$, we only care about the $L^{\infty}-$norm of $f$, though we can assume that everything involved  have higher derivatives.  Multiplying (\ref{equ heat wrt Riemannian metrics}) by $\sqrt{detg_{ij}}$, we get 
  \begin{equation}\label{equ correspondence heat wrt Riemannian metrics}
  \sqrt{detg_{ij}} \frac{\partial u}{\partial t}-\frac{\partial}{\partial x_{i}}(g^{ij}\sqrt{detg_{ij}}\frac{\partial u}{\partial x_{j}})  =F= f\sqrt{detg_{ij}}.
\end{equation}  
Let $a_{0}=\sqrt{detg_{ij}} $ and $a^{ij}=g^{ij}\sqrt{detg_{ij}}$,  (\ref{equ heat wrt Riemannian metrics}) is a special case of  (\ref{equ considered by FS}) and  equation (D) in page 89 of \cite{FS}. Suppose  $detg_{ij}$ is uniformly bounded, the  $L^{\infty}-$norm of $f$ is equivalent to the   $L^{\infty}-$norm of $F$, thus it makes no difference for the H\"older estimate. 

 Our main  observation (and a one sentence proof of Theorem \ref{thm Harnack inequality parabolic}) is that \textbf{when $a_{0}$ depends on time and  $\frac{\partial \log {a_{0}}}{\partial t}$ is bounded from above,   the general energy estimates  are still true} (Lemma \ref{lem: parabolic inequality holds in the right way}).    By the  proof in  \cite{FS}, these energy estimates  imply the main growth theorem (Theorem 5.3)  in \cite{FS}.   Moreover, by an idea in \cite{LSU},  Theorem 5.3 in \cite{FS}  directly implies the H\"older continuity of solutions, without involving the Harnack inequality in Theorem 1.5 of \cite{FS}.  We believe these are  known by experts. When $g_{t}$ is a  Ricci flow, the upper bound on $\frac{\partial \log {a_{0}}}{\partial t}$    means 
 \begin{equation}\label{equ condition upper bound of volume form derivative}
 \frac{\partial }{\partial t}dvol_{g_t}\leq Kdvol_{g_t}
 \end{equation} for some constant $K>0$, where $dvol_{g_t}$ is the evolving volume form.  The $K$ is actually a lower bound for the scalar curvature of $g_{t}$. Fortunately, the scalar curvature is usually bounded from below along Ricci flows without any additional condition, see \cite{Cao} (page 5) and \cite{Hamilton}.
 
 The simplest version of our main theorem is stated as follows. Let $Y=(y,s)$ be a space-time point, and $C_{r}(Y)=B_{y}(r)\times (s-r^{2},s)$ be the parabolic cylinder centred at $Y$ with  radius $r$ [$B_{y}(r)$ is the usual $m-$dimensional Euclidean ball].
\begin{thm}\label{thm Harnack inequality parabolic} Suppose  $u\in C^{\infty}[C_{r}(Y)]$ solves the following equation (or the metric heat equation (\ref{equ heat wrt Riemannian metrics}) via the correspondence in (\ref{equ correspondence heat wrt Riemannian metrics}) )  in the classical sense 
\begin{equation}\label{equ considered by FS}
a_{0}\frac{\partial u}{\partial t}-\frac{\partial }{\partial x_{j}}(a^{ij}\frac{\partial u}{\partial x_{i}})=f, 
\end{equation}
where $a_{0}$, $a^{ij}$ ($1\leq i,j\leq m$) are  space-time smooth functions. Suppose 
\begin{equation}\label{equ crucial upper bound on the vol form derivative}\frac{1}{K}\leq a_{0}\leq K,\ \frac{\partial \log a_{0}}{\partial t}\leq K,\ \frac{I}{K}\leq a^{ij}\leq KI. 
\end{equation}
Then there exist constants $\alpha(m,K)\in (0,1)$ and $N(m,K)$ such that 
\[ r^{\alpha}{[u]}_{\alpha, C_{\frac{r}{2}}(Y)}+|u|_{L^{\infty}[C_{\frac{r}{2}}(Y)]}\leq N(\frac{|u|_{L^{1}[C_{r}(Y)]}}{r^{m+2}}+r^{2}|f|_{L^{\infty}[C_{r}(Y)]}).\]
\end{thm}
\begin{rmk}The $[\ \cdot \ ]_{\alpha}$ is the \textbf{parabolic} H\"older semi-norm of exponent $\alpha$ (see (4.1) in \cite{Lieberman} for  definition). \textbf{Theorem \ref{thm Harnack inequality parabolic} can be generalized to heat equations of K\"ahler-metrics  with conic singularities along normal-crossing divisors} (Theorem \ref{thm conic Harnack inequality parabolic}). \textbf{We only prove Theorem \ref{thm conic Harnack inequality parabolic}},   the proof of Theorem \ref{thm Harnack inequality parabolic} is the same (by 
discarding the  necessary techniques for the conic singularities, see Claim \ref{equ L2 norm of gradient of Berdtsson cutoff function vanishes} for example).   We check routinely in Section \ref{Section proof of the Growth theorem} that \textbf{the only usage of the equation in proving Theorem \ref{thm growth theorem} is  the energy estimate}  (for all  test functions, at  all levels and in all scales). \end{rmk}
\begin{rmk}When $\frac{\partial \log a_{0}}{\partial t}$ is not uniformly bounded from above (while the other conditions in Theorem \ref{thm Harnack inequality parabolic} hold true), the above uniform H\"older estimate does not hold in general. We refer the interested readers to   the beautiful example constructed by Chen-Safonov  (Theorem 4.1 and 4.2 in \cite{ChenSafonov}). 
\end{rmk}
\textbf{Acknowledgement}: The author is grateful to Gong Chen for valuable communications. 

\section{The more general version of Theorem \ref{thm Harnack inequality parabolic} in K\"ahler geometry involving conic singularity}

 In K\"ahler geometry setting, Theorem \ref{thm Harnack inequality parabolic}  holds even when the metrics possess conic singularities along analytic hyper surfaces. To state the result,  we first give a geometric formulation  following \cite{Guenancia}. Given a closed K\"ahler manifold $M$ and a divisor $D=\Sigma_{j=1}^{N}2\pi(1-\beta_{j})D_{j}$, where each  $D_{j}$ is an irreducible hyper surface and  may have self-intersection, suppose  $D$ has (no worse than) normal crossing singularities i.e there is  an open cover of $supp D$ by neighbourhoods $\U_{i}$ such that in each $\U_{i}$, $supp D\cap \U_{i}=\{z_{1} z_{2}z_{3}\cdot\cdot\cdot z_{k}=0\}$, where $k\leq n$ and $z_{1}...z_{n}$  are holomorphic coordinate functions  in $\U_{i}$.  A K\"ahler metric $g$ (defined away from $supp D$) is said to be a weak-conic  metric with quasi-isometric constant $K$, iff  it's H\"older-continuous away from $supp D$ and in each $\U_{i}$,   
     \begin{equation}\label{equ gbeta k}
\frac{g^{k}_{\beta}}{K}\leq g\leq Kg^{k}_{\beta}\ (\textrm{quasi-isometric}),\ g^{k}_{\beta}=\Sigma_{j=1}^{k}\frac{\beta_{j}^{2}}{|z_{j}|^{2-2\beta_{j}}}dz_{j}\otimes d\bar{z}_{j}+\Sigma_{j=k+1}^{n}dz_{j}\otimes d\bar{z}_{j}.
\end{equation}
$g_{\beta}^{k}$ is one of the 2  model metrics  on $\C^{n}$ we work with, and in this local setting we abuse notation by denoting $supp D$ as $D$.
 
Similarly, a K\"ahler metric $g$ is called a $\epsilon$-nearly-conic  metric with quasi-isometric constant $K$, iff  it's H\"older-continuous over the whole $M$ (across $supp D$) and in each $\U_{i}$, 
 \begin{equation}\label{equ gbeta k epsilon}
\frac{g^{k}_{\beta,\epsilon}}{K}\leq g\leq Kg^{k}_{\beta,\epsilon},\ 
g^{k}_{\beta,\epsilon}=\Sigma_{j=1}^{k}\frac{\beta_{j}^{2}}{(|z_{j}|^2+\epsilon^2)^{1-\beta_{j}}}dz_{j}\otimes d\bar{z}_{j}+\Sigma_{j=k+1}^{n}dz_{j}\otimes d\bar{z}_{j},\ \epsilon>0.
\end{equation}
  We recall the well known intrinsic  polar coordinates of $g^{k}_{\beta}$. Let $\xi_{j}=r_{j}e^{\sqrt{-1}\theta_{j}},\ r_{j}=|z_{j}|^{\beta_{j}},\ 1\leq j\leq k$.  In these polar  coordinates 
the model cone  $g^{k}_{\beta}$  is equal to  
\[g^{k}_{\beta}=\Sigma_{j=1}^{k}(dr_{j}^2+\beta_{j}^2r_{j}^2d\theta_{j}^2)+\Sigma_{j=k+1}^{n}dz_{j}\otimes d\bar{z}_{j},\]
and it's quasi-isometric to  the Euclidean metric i.e 
\begin{equation}\label{equ conic model metric quasi-isometric to Euclidean metric}
(\min_{j}\beta_{j}^{2})g_{E}\leq g^{k}_{\beta}\leq g_{E},\ g_{E}=\Sigma_{j=1}^{k}(dr_{j}^2+r_{j}^2d\theta_{j}^2)+\Sigma_{j=k+1}^{n}dz_{j}\otimes d\bar{z}_{j}.
\end{equation} 
This is important because we want to take advantage of the rescaling and translation invariance of the Euclidean metric. 

Similarly, we also have intrinsic polar coordinates for $g^{k}_{\beta,\epsilon}$. Let $s_{j}$ be the solution to 
\begin{equation}\label{equ ODE for s}
\frac{ds_{j}}{d\rho_{j}}=\frac{\beta_{j}}{(\rho_{j}^2+\epsilon^2)^{\frac{1-\beta_{j}}{2}}},\ s_{j}(0)=0,\ \rho_{j}=|z_{j}|.
\end{equation}
Then $\xi_{j}=s_{j}e^{\sqrt{-1}\theta_{j}},\ 1\leq j\leq k$ defines the polar coordinates of $g^{k}_{\beta,\epsilon}$. By Lemma 4.3 in \cite{WYQWF}, in these coordinates we have 
\begin{equation}
g^{k}_{\beta,\epsilon}=\Sigma_{j=1}^{k}(ds_{j}^2+a_{j,\epsilon}s_{j}^2d\theta_{j}^2)+\Sigma_{j=k+1}^{n}dz_{j}\otimes d\bar{z}_{j},\  \beta_{j}^{2}<a_{j,\epsilon}\leq 1. 
\end{equation}
Hence $g^{k}_{\beta,\epsilon}$ is also quasi-isometric to the Euclidean metric in its polar coordinate i.e 
\begin{equation}\label{equ nearly-conic model metric quasi-isometric to Euclidean metric}
(\min_{j}\beta_{j}^{2})g_{E}\leq g^{k}_{\beta,\epsilon}\leq g_{E},\ g_{E}=\Sigma_{j=1}^{k}(ds_{j}^2+s_{j}^2d\theta_{j}^2)+\Sigma_{j=k+1}^{n}dz_{j}\otimes d\bar{z}_{j}.
\end{equation}   
   
Unless specified (via a parentheses or a sub-symbol),  the constants  $N$ and $C$ in this article  depend  on (at most) $n,K,M,D,\beta_{j}'s$, and  the open cover $\cup_{i}\U_{i}$. They \textbf{don't depend  on $\epsilon$}. Different $N's$ could be different.  \textbf{The real dimension  is $m=2n$ in the K\"ahler setting.} 
\begin{thm}\label{thm conic Harnack inequality parabolic}Let $\epsilon \in [0,1]$. 

Part $I$ (local estimate):  Suppose  $g_t$ is
a time-differentiable family of weak-conic K\"ahler metrics or of $\epsilon$-nearly-conic  metrics, which is defined over a parabolic cylinder $C_{r}(Y)$ in $\C^{n}$ under a polar coordinate as below    (\ref{equ gbeta k epsilon}) or (\ref{equ ODE for s}), respectively. Suppose   the  quasi-isometric constant of $g_t$ is  $K$, 
\begin{equation}\label{equ conditions of the metrics}  \frac{\partial }{\partial t}dvol_t\leq Kdvol_t.
\end{equation}
and  $u$ is  a bounded weak solution to  
\begin{equation}\label{equ conic heat equ}\frac{\partial u}{\partial t}
=\Delta_{g_{t}}u+f\ \textrm{over}\ C_{r}(Y).
\end{equation}
 Then there exists $\alpha(n,,\beta,K)\in (0,1)$ and $N(n,\beta,K)$ such that 
\[ r^{\alpha}{[u]}_{\alpha, C_{\frac{r}{2}}(Y)}+|u|_{L^{\infty},C_{\frac{r}{2}}(Y)}\leq N(\frac{|u|_{L^{1}[C_{r}(Y)]}}{r^{2n+2}}+r^{2}|f|_{L^{\infty}[C_{r}(Y)]}).\]

Part $II$ (global estimate) In the setting of (\ref{equ gbeta k}) and paragraph above it, suppose all the conditions in part $I$ hold globally on  $M\times [0,T]$. Then for all $t_{0}\in (0,T)$, there exists an $\alpha(n,\beta,K)$ and $C_{t_{0}}(n,\beta,K)$ such that 
\[{[u]}_{\alpha,M\times [t_{0},T]}+|u|_{L^{\infty}(M\times [t_{0},T])}\leq C_{t_{0}}(|u|_{L^{1}(M\times [0,T])}+|f|_{L^{\infty}(M\times [0,T])}).\]
\end{thm}
\begin{rmk} When the divisor is smooth, a weaker version of  this H\"older estimate is  in section 4 of \cite{WYQWF}. We hope it's still somewhat valuable to present the proof  separately  here.  The ${[u]}_{\alpha}$ is the usual parabolic H\"older semi-norm with respect to $g^{k}_{\beta}$ ($g^{k}_{\beta,\epsilon}$)  [see (\ref{equ conic model metric quasi-isometric to Euclidean metric})]. An important point is that H\"older continuity with respect to the distance of $g^{k}_{\beta}$ ($g^{k}_{\beta,\epsilon}$) is equivalent to H\"older continuity in the usual sense in holomorphic coordinates (apart from a difference of H\"older exponents). We refer interested readers to Lemma 4.4 in \cite{WYQWF}. Please see Definition \ref{Def weak solution} for definition of weak solutions (replace $SC_{r}$ by the underlying domain).
\end{rmk}
\begin{rmk}Using the K\"ahler structure, equation  (\ref{equ conic heat equ}) can be written as both  divergence and non-divergence form. We  expect that Theorem \ref{thm conic Harnack inequality parabolic} still holds without condition (\ref{equ conditions of the metrics}).
\end{rmk}

\section{Proof of the main results assuming  Theorem \ref{thm growth theorem}.}
From now on (and in the subsequent sections), we work in the polar coordinates in (\ref{equ conic model metric quasi-isometric to Euclidean metric}) and (\ref{equ nearly-conic model metric quasi-isometric to Euclidean metric}). \textbf{In this coordinate, we don't see the conic singularity} (except that the coefficients of the equations and solutions are not defined on $D$). Let $C^{0}_{r}$ denote $C_{r}(y,s-3r^{2})$ (see the paragraph above Theorem \ref{thm Harnack inequality parabolic}).
 \begin{proof}{of Theorem  \ref{thm Harnack inequality parabolic}, \ref{thm conic Harnack inequality parabolic}:}  We only prove (part $I$ of) Theorem \ref{thm conic Harnack inequality parabolic} as mentioned at the end of the introduction. Notice that $y$ does not have to be in $supp D$ (as long as integration by parts is true, see proof of Lemma \ref{lem: parabolic inequality holds in the right way}). By the interior $L^{\infty}-$estimate in Proposition \ref{prop Moser iteration equivalent to growth lemma} \textbf{which holds for every cylinder and every sub-solution}, it suffices to show the H\"older norm is bounded by the $L^{\infty}-$norm i.e.
 \begin{equation}\label{equ in proof of Main Thm }r^{\alpha}{[u]}_{\alpha, C_{\frac{r}{2}}(Y)}\leq N(|u|_{L^{\infty}[C_{r}(Y)]}+r^{2}|f|_{L^{\infty}[C_{r}(Y)]}).
\end{equation}
 By Lemma 4.6 in \cite{Lieberman}, it suffices to show the oscillation decays \textbf{for every cylinder $C_{2r}$ and every sub-solution $u$} i.e.
  \begin{equation}\label{equ osc decreasing}
osc_{C_{r}}u\leq (1-b)osc_{C_{2r}}u+4r^{2}|f|_{0,C_{2r}}\,\ b=b(n,\beta,K)>0. 
\end{equation}
By rescaling and translation invariance,  it suffices to assume $r=1,\ s=0$. By adding a constant, it suffices to assume $0\leq u\leq h$, where  $h\triangleq osc_{C_{2}}u$. As in \cite{LSU},  one of the following must hold:\\
 Case 1:  $|\{u> \frac{h}{2}\}\cap C^{0}_{1}|\leq \frac{|C^{0}_{1}|}{2}$\ ;\ \ \ \ \
Case 2:  $|\{u<\frac{h}{2}\}\cap C^{0}_{1}|\leq  \frac{|C^{0}_{1}|}{2}$.

We only prove (\ref{equ osc decreasing}) in  Case 1 in detail. Case 2 is similar by applying the proof in Case 1 to  $h-u$. Consider $\bar{u}=u-t|f|_{0,C_{2r}}$ [$t\in (-4,0)$]. Then 
  \begin{equation}\label{equ 2 in proof of Main Thm}
  \frac{\partial \bar{u}}{\partial t}-\Delta_{g}\bar{u}\leq 0,\  0\leq \bar{u}\leq h+4|f|_{0,C_{2}}.\
  \textrm{Moreover},\ \bar{u}>\frac{h}{2}+4|f|_{0,C_{2}}\Rightarrow u>\frac{h}{2}. 
  \end{equation}
Hence the assumption of Case 1 implies  
$$|\{\bar{u}\geq \frac{h}{2}+4|f|_{0,C_{2}}\}\cap C^{0}_{1}|\leq |\{u> \frac{h}{2}\}\cap C^{0}_{1}| \leq \frac{|C^{0}_{1}|}{2}.$$ Then Theorem \ref{thm growth theorem} (applied to $\bar{u}-\frac{h}{2}-4|f|_{0,C_{2}}$), (\ref{equ 2 in proof of Main Thm}), and the above inequality imply that there exists $a(n,\beta,K)>0$ such that 
\begin{eqnarray*}\sup_{C_{1}}(\bar{u}-\frac{h}{2}-4|f|_{0,C_{2}})
\leq  (1-a)\sup_{C_{2}}(\bar{u}-\frac{h}{2}-4|f|_{0,C_{2}})\leq \frac{(1-a)h}{2}.
\end{eqnarray*}
Then $osc_{C_{1}}\leq \sup_{C_{1}}u \leq \sup_{C_{1}}\bar{u}\leq (1-\frac{a}{2})h+ 4|f|_{0,C_{2}}$. The proof of (\ref{equ osc decreasing}) (under the normalization conditions below it)  is complete.
\end{proof}
\section{Energy inequalities}
We follow closely the definitions and tricks in \cite{FS}, the point is that \textbf{they work equally well in the presence of  conic singularity} (Definitions \ref{Def slant cylinder}, \ref{Def weak solution}, and \ref{Def the cutoff trick by G} ). \textbf{The functions and integrations are all defined away from the divisor $D$} [see the content below (\ref{equ gbeta k})]. \textbf{If the notation of a function space does not involve $D$, we mean the space satisfies the indicated asymptotic property} (which should be clear from the context).  The sets and (slant) cylinders are standard ones minus $D$. This does not affect any measure theory, integration, or technique in this article, because the space wise co-dimension of $D$ is $2$. For the proof of Theorem \ref{thm Harnack inequality parabolic}, we don't have to chop of any singularity.
\begin{Def}\label{Def slant cylinder} A slant cylinder $SC_{r}(y_{0},y_{1},T_{0}, T_{1})$, which we abbreviate in most the time as $SC_{r}$, is the following set
\begin{equation}
SC_{r}\triangleq \{x|\ |x-y(t)|<r,\ T_{0}<t\leq T_{1}\}, 
\end{equation}
where $y(t)=y_{0}+\frac{(t-T_{0})(y_{1}-y_{0})}{T_{1}-T_{0}}$.  \textrm{When $y_{1}=y_{0}$, $SC_{r}$ is just the usual cylinder $C_{r}$} defined above Theorem \ref{thm Harnack inequality parabolic}. We define   $l\triangleq \frac{r(y_{1}-y_{0})}{T_{1}-T_{0}}$ as the parabolic slope of $SC_{r}$. $l$ is invariant under 
\begin{itemize}\item the usual parabolic rescaling (linear multiplication on $y_{0},y_{1},r$ and quadratic multiplication on $T_{0},T_{1}$ by the same factor),
\item  the space-wise translation (on $y_{0},y_{1}$ by the same displacement), 
\item and the time-wise translation (on $T_{0}$ and $T_{1}$ by the same displacement).
\end{itemize}
\end{Def}
\begin{Def}\label{Def weak solution} We say $u$ is a weak sub solution to 
\begin{equation}\label{equ sub heat equation of g}
\frac{\partial u}{\partial t}-\Delta_{g}u\leq 0,
\end{equation}
in a slant cylinder $SC_{r}$ if 
\begin{enumerate}
\item $u\in C^{2+\alpha,1+\frac{\alpha}{2}} \{SC_{r}\setminus D\times [T_{0},T_{1}]\}\cap L^{\infty}(SC_{r})$;
\item Inequality (\ref{equ sub heat equation of g}) holds over 
$SC_{r}\setminus D\times [T_{0},T_{1}]$ in the classical sense. 
\end{enumerate}
We call a function $\eta$ (defined in  any bounded space-time domain  $\Omega\in \C^{n}\times (-\infty,\infty)$) tame if $\eta\in C^{1,1} \{\Omega \setminus D\times [T_{0},T_{1}]\}\cap L^{\infty}(\Omega)$ and the following holds. 
\begin{equation}
\frac{\partial \eta}{\partial t}\in L^{1}(\Omega),\ \nabla \eta\in L^{2}(\Omega).
\end{equation}
\end{Def}
\begin{rmk}The $L^{\infty}(SC_{r})$-requirement in Definition \ref{Def weak solution} is crucial, and is  the only global condition. It guarantees (\ref{equ sub heat equation of g}) holds across the singularity  in the sense of integration by parts.
\end{rmk}
\begin{Def}\label{Def the cutoff trick by G} Exactly as in Corollary 2.3 in \cite{FS}, we define 
the cutoff function of $u$ as 
\begin{equation}\label{equ the cutoff trick by G}
u_{\epsilon}=G(u),
\end{equation}
where $G$ is a function with one variable such that $G(u)=0$ when $u\leq \epsilon$, $G(u)=u+G(2\epsilon)-2\epsilon$ when $u\geq 2\epsilon$, and $G,G^{\prime},G^{\prime \prime}\geq 0$. Consequently, we have 
\begin{equation}
G(2\epsilon)\leq \epsilon\ \textrm{and}\ \max\{u-2\epsilon,0\}\leq u_{\epsilon}\leq \max\{u-\epsilon,0\}.
\end{equation}

The most important feature of $u_{\epsilon}$ is that, suppose $u$ is sub solution to (\ref{equ sub heat equation of g}), so is $u_{\epsilon}$ i.e 
\begin{equation}
\frac{\partial u_{\epsilon}}{\partial t}-\Delta_{g}u_{\epsilon}\leq 0. 
\end{equation}

 $u_{\epsilon}$ can be understood as   the smoothing of $u^{+}$ (non-negative part of $u$). We note that in the classical case, $u^{+}$ is a sub solution (in proper sense) if $u$ is. The above smoothing is point wise, thus works in the presence of conic singularities.
\end{Def}

\begin{lem}\label{lem: parabolic inequality holds in the right way} Under the same assumptions in Part $I$ of Theorem \ref{thm conic Harnack inequality parabolic} (for any $r$), suppose $u$ is a non-negative weak solution to (\ref{equ sub heat equation of g}) in the sense of Definition \ref{Def weak solution} in a slant cylinder $SC_{r},\ r\leq \frac{1}{100 n}$. Then for any non-negative tame  function $\eta$ which is compactly supported in $SC_{r}$ space wisely, we have 
\begin{eqnarray}\label{equ weak solution identity}
& &\int_{\C^{n}} u\eta^{2} dV_{g}|^{t_{2}}_{t_{1}}+\int^{t_{2}}_{t_{1}}\int_{\C^{n}} <\nabla_{g}u,\nabla_{g}\eta^{2}>dV_{g}ds\nonumber
\\&\leq & \int^{t_{2}}_{t_{1}}\int_{\C^{n}} u\frac{\partial \eta^{2}}{\partial t} dV_{g}ds+K\int^{t_{2}}_{t_{1}}\int_{\C^{n}}u \eta^{2} dV_{g}ds.
\end{eqnarray}

Moreover, we have 
\begin{eqnarray}\label{equ weak solution energy estimate}
& &\int_{\C^{n}} u^2\eta^2 dV_{g}|^{t_{2}}_{t_{1}}+\frac{1}{2}\int^{t_{2}}_{t_{1}}\int_{\C^{n}} |\nabla_{g}u|^2 \eta^{2} dV_{g}ds
\\&\leq & \int^{t_{2}}_{t_{1}}\int_{\C^{n}} u^2\frac{\partial \eta^2}{\partial t} dV_{g}ds+(2K+200)\int^{t_{2}}_{t_{1}}\int_{\C^{n}}u^2(\eta^2+|\nabla_{g}\eta|^2) dV_{g}ds,\nonumber
\end{eqnarray}
and therefore 
\begin{equation}\label{equ parabolic energy is bounded}
\int_{\Omega} |\nabla_{g}u|^2dV_{g}ds<+\infty,\ \textrm{for any (parabolic) compact sub-domain}\ \Omega\ \textrm{of}\ SC_{r}.
\end{equation}
\end{lem}
\begin{rmk} By the same proof, the energy estimate of (\ref{equ considered by FS}) is similar.
\end{rmk}
\begin{proof}{of Lemma \ref{lem: parabolic inequality holds in the right way}:} Let $r_{i}$  be  the distance function to the smooth hyper surface  $D_{i}$ . We consider Berdtsson's cutoff function $\psi_{i,\epsilon}=\psi(\epsilon\log(-\log r_{i}))$, $\psi$ is the standard cutoff function such that $\psi(x)\equiv 1$ when $x\leq \frac{1}{2}$, and $\psi(x)\equiv 0$ when $x\geq \frac{4}{5}$.  Then 
\begin{equation}\label{equ psi epsilon is a cutoff function}
\psi_{i,\epsilon}\equiv 0\ \textrm{when}\ r_{i}\leq  e^{-e^{\frac{4}{5\epsilon}}};\ \psi_{i,\epsilon}\equiv 1\ \textrm{when}\ r_{i}\geq  e^{-e^{\frac{1}{2\epsilon}}}.
\end{equation}

Let $\psi_{\epsilon}=\prod_{i=1...n}\psi_{i,\epsilon}$, the following claim is true.
\begin{clm}\label{equ L2 norm of gradient of Berdtsson cutoff function vanishes}
\begin{equation}\label{equ W12 norm of psi tends to 0}
\lim_{\epsilon \rightarrow 0}|\nabla_{E}\psi_{\epsilon}|_{L^{2}(B(\frac{1}{2}))}=0.
\end{equation}
\end{clm}
The proof of Claim \ref{equ L2 norm of gradient of Berdtsson cutoff function vanishes} is elementary. We only  verify it for  $\frac{\partial \psi_{\epsilon}}{\partial r_{1}}$, the other directional derivatives are similar. We compute $\frac{\partial \psi_{\epsilon}}{\partial r_{1}}=-\psi^{\prime}\frac{\epsilon}{r_{1} \log r_{1}}\prod_{i\neq 1}\psi_{i,\epsilon}$. Hence in poly-cylindrical coordinates we find \begin{eqnarray}
& &\int_{B(\frac{1}{2})}|\frac{\partial \psi_{\epsilon}}{\partial r_{1}}|^{2}dvol_{E}\nonumber
\leq  C\epsilon^{2} \int^{\frac{1}{2}}_{0}\frac{1}{r_{1}(\log r_{1})^2}dr_{1}\leq C\epsilon^{2}. 
\end{eqnarray}
. 

We first prove (\ref{equ weak solution energy estimate}). By definition we have $\lim_{\epsilon \rightarrow 0}\psi_{\epsilon}=1$ everywhere except on $supp D$. We  multiply both hand sides of (\ref{equ sub heat equation of g}) by $u\eta^2 \psi^2_{\epsilon}$, then integrate by parts and integrate with respect to time, we obtain 
\begin{eqnarray}& &\frac{1}{2}\int_{\C^{n}}u^2\eta^2\psi^2_{\epsilon}dV_{g}|^{t_{2}}_{t_{1}}
+\int^{t_{2}}_{t_{1}}|\nabla_{g}u|^2\eta^{2}\psi_{\epsilon}^2dV_{g}ds
\\&\leq & \frac{1}{2}\int^{t_{2}}_{t_{1}} \int_{\C^{n}}u^2\eta^2\psi^2_{\epsilon}\frac{\partial dV_{g}}{\partial t}ds
-2\int^{t_{2}}_{t_{1}}\int_{\C^{n}}<\nabla_{g}u,\nabla_{g}\eta>u\eta\psi^{2}_{\epsilon}dV_{g}ds\nonumber
\\& &+\frac{1}{2}\int^{t_{2}}_{t_{1}} \int_{\C^{n}}u^2\frac{\partial \eta^2}{\partial t}\psi^2_{\epsilon}dV_{g}ds-2\int^{t_{2}}_{t_{1}}\int_{\C^{n}}<\nabla_{g}u,\nabla_{g}\psi_{\epsilon}>u\eta^2\psi_{\epsilon}dV_{g}ds.\nonumber
\end{eqnarray}

Using Cauchy-Schwartz inequality we deduce that 
\begin{eqnarray}
& &|2\int^{t_{2}}_{t_{1}}\int_{\C^{n}}<\nabla_{g}u,\nabla_{g}\psi_{\epsilon}>u\eta^2\psi_{\epsilon}dV_{g}ds|
\\& \leq & \frac{1}{100}\int^{t_{2}}_{t_{1}}\int_{\C^{n}}|\nabla_{g}u|^2\psi_{\epsilon}^2\eta^2dV_{g}ds+100\int^{t_{2}}_{t_{1}}\int_{\C^{n}}u^2\eta^{2}|\nabla_{g}\psi_{\epsilon}|^2dV_{g}ds.\nonumber
\end{eqnarray}
Similarly we have 
\begin{eqnarray}
& &|2\int^{t_{2}}_{t_{1}}\int_{\C^{n}}<\nabla_{g}u,\nabla_{g}\eta>u\eta\psi^2_{\epsilon}dV_{g}ds|
\\& \leq & \frac{1}{100}\int^{t_{2}}_{t_{1}}\int_{\C^{n}}|\nabla_{g}u|^2\psi_{\epsilon}^2\eta^2dV_{g}ds+100\int^{t_{2}}_{t_{1}}\int_{\C^{n}}u^2\psi_{\epsilon}^2|\nabla_{g}\eta|^2dV_{g}ds.\nonumber
\end{eqnarray}

Notice that by (\ref{equ condition upper bound of volume form derivative}) we have \begin{equation}\int^{t_{2}}_{t_{1}} \int_{\C^{n}}u^2\eta^2\psi^2_{\epsilon}\frac{\partial dV_{g}}{\partial t}ds
\leq K\int^{t_{2}}_{t_{1}} \int_{\C^{n}}u^2\eta^2\psi^2_{\epsilon} dV_{g}ds.
\end{equation}

Then \begin{eqnarray}\label{equ energy estimate main big step}& &\frac{1}{2}\int_{\C^{n}}u^2\eta^2\psi^2_{\epsilon}dV_{g}|^{t_{2}}_{t_{1}}
+\int^{t_{2}}_{t_{1}}\int_{\C^{n}}|\nabla_{g}u|^2\eta^{2}\psi_{\epsilon}^2dV_{g}ds
\\&\leq & K\int^{t_{2}}_{t_{1}} \int_{\C^{n}}u^2\eta^2\psi^2_{\epsilon} dV_{g}ds+\frac{1}{2}\int^{t_{2}}_{t_{1}} \int_{\C^{n}}u^2\frac{\partial \eta^2}{\partial t}\psi^2_{\epsilon}dV_{g}ds\nonumber
\\& & +\frac{1}{100}\int^{t_{2}}_{t_{1}}\int_{\C^{n}}|\nabla_{g}u|^2\psi_{\epsilon}^2\eta^2dV_{g}ds+100\int^{t_{2}}_{t_{1}}\int_{\C^{n}}u^2\psi_{\epsilon}^2|\nabla_{g}\eta|^2dV_{g}ds\nonumber
\\& &+\frac{1}{100}\int^{t_{2}}_{t_{1}}\int_{\C^{n}}|\nabla_{g}u|^2\psi_{\epsilon}^2\eta^2dV_{g}ds+100\int^{t_{2}}_{t_{1}}\int_{\C^{n}}u^2\eta^{2}|\nabla_{g}\psi_{\epsilon}|^2dV_{g}ds.\nonumber
\end{eqnarray}
We note that Definition \ref{Def weak solution} requires $u\in L^{\infty}$, then (\ref{equ W12 norm of psi tends to 0}) implies 
 \begin{equation}
 \lim_{\epsilon \rightarrow 0}100\int^{t_{2}}_{t_{1}}\int_{\C^{n}}u^2\eta^{2}|\nabla_{g}\psi_{\epsilon}|^2dV_{g}ds=0.
 \end{equation}
Let $\epsilon\rightarrow 0$ in (\ref{equ energy estimate main big step}), the proof of (\ref{equ weak solution energy estimate}) and (\ref{equ parabolic energy is bounded}) is complete. 

 Multiplying  both hand sides of (\ref{equ sub heat equation of g}) by $\eta \psi_{\epsilon}$ and   integrating by parts over space-time,  (\ref{equ weak solution identity}) is proved similarly.
 \end{proof}

By the same proof as Lemma \ref{lem: parabolic inequality holds in the right way} (with Berndtsson's cutoff function),  the Sobolev embedding theorem is true.
\begin{lem}\label{lem Sobolev}(Sobolev Embedding)
Given a function $u\in C^{1} \{B\setminus D\}\cap L^{\infty}(B)$, for any cutoff function $\eta\in C_{0}^{1}(B)$, the following holds. 
\[ (\int_{B}|\eta u|^{\frac{2n}{2n-1}}dV_{E})^{\frac{2n-1}{2n}}\leq N(\beta,n)\int_{B}|\nabla (\eta u)|dV_{E}.\]
\end{lem}
\begin{proof} It's true when $\int_{B}|\nabla (\eta u)|dV_{E}=\infty$. When $\int_{B}|\nabla (\eta u)|dV_{E}<\infty$, using $\psi_{\epsilon}$, Claim \ref{equ L2 norm of gradient of Berdtsson cutoff function vanishes}, and the same proof as in Lemma \ref{lem: parabolic inequality holds in the right way},  $\eta u$ belongs to $W^{1,1}(B)$ in the usual sense. Then it follows  from the usual Sobolev-inequality.\end{proof}
\begin{rmk}The $N(\beta,n)$ above does not depend on the radius or center of the ball. The only place where we use the Sobolev embedding is (\ref{equ the only place where Sobolev imbedding is used explicitly}).
\end{rmk}
\section{Proof of Theorem \ref{thm growth theorem} by   energy inequalities\label{Section proof of the Growth theorem}}\subsection{Growth Lemma}

\begin{prop}(Growth Lemma)\label{prop growth thm 1} Suppose $u$ is a weak sub solution to (\ref{equ sub heat equation of g}) in a  cylinder $C_{2r}(Y)$. Then there exists a $\mu_{2}(n,\beta,K)>0$ such that  
\begin{equation}\label{equ 1 in Growth Lem statement}
\frac{|\{u>0\}\cap C_{2r}(Y)|}{|C_{2r}(Y)|}\leq \mu_{2}\ \textrm{implies}\
\sup_{C_{r}}u\leq \frac{1}{2}\sup_{C_{2r}}u^{+}.
\end{equation}

\end{prop}
\begin{proof}[Proof of Proposition \ref{prop growth thm 2}:] The proof is formally the same as Lemma 4.1 in \cite{FS}. Since condition (\ref{equ condition upper bound of volume form derivative}) is involved, we still give a detailed proof for the reader's convenience. \textbf{The point is to show that we don't need  more on the equation than the energy estimates of sub-solutions} (Lemma \ref{lem: parabolic inequality holds in the right way} and the proof of it). The constants $N$ in this proof only depend on $n,\beta,K$.

 By rescaling invariance of the sub-equation (\ref{equ sub heat equation of g}),  it suffices to assume $r=1$ and  $\sup_{C_{r}} u=1$. We let $\mu_{2}$ be small enough.    It sufficies to prove that for all $Z\notin D$ and $Z\in C_{1}(Y)=C_{1}$, under the condition \begin{equation}
\frac{|\{u>0\}\cap C_{1}(Z)|}{|C_{1}(Z)|}\leq \frac{2^{2n+2}|\{u>0\}\cap C_{2}(Z)|}{|C_{2}(Z)|}\leq 2^{2n+2}\mu_{2}\triangleq \mu_{1},
\end{equation}
 the following estimate holds \begin{equation}
u(Z)\leq \frac{1}{2}.
\end{equation} 

We only need to apply Lemma \ref{lemma L1 estimate} ( (3.8) in page 33 of \cite{FS}).  Using exactly the induction argument from the last line of page 99 to line 16 of page 100 in \cite{FS} (only involving  Lemma \ref{lemma L1 estimate}), we deduce for any integer $j\geq 0$, for some $N(n,\beta,K)$, the following estimate holds when $N\mu_{1}^{\frac{1}{2n+2}}<\frac{1}{2}$.  
\begin{equation}\label{equ decay estimate of super level set}
|\{u>\frac{1-\rho}{2}\}\cap C_{\rho}(Z)|\leq \mu_{1}\rho^{2n+2}|C_{\rho}(Z)|,\ \rho=2^{-j}. 
\end{equation}
Since $Z\notin D$, (\ref{equ decay estimate of super level set}) directly implies that $u(Z)\leq \frac{1}{2}$. Were this not true, $u(Z)>\frac{1}{2}$ implies that there exists dyadic $\rho_{0}$ small enough such that $C_{\rho_{0}}(Z)$ does not touch the singularity $D$, and  $u>\frac{1}{2}$ over $C_{\rho_{0}}(Z)$. This contradicts (\ref{equ decay estimate of super level set}).



\end{proof}
\begin{lem}\label{lemma L1 estimate} Under the same setting as in Proposition \ref{prop growth thm 1} and its proof above, for any constant $A\geq 0$, we have $$\int_{C_{\frac{\rho}{2}}(Z)}(u-A)_{+}dV_{E}ds\leq 
\frac{N}{\rho}|\{u>A\}\cap C_{\rho}(Z)|^{1+\frac{1}{2n+2}}.$$
\end{lem}
\begin{proof}{of Lemma \ref{lemma L1 estimate}:} By linearity and rescaling invariance of the sub equation (\ref{equ sub heat equation of g}), without loss of generality we can assume $A=0$ and $\rho=1$ (note $u\leq 1$). Denote the set $\{(u>0)\cap C_{1}(Z)\}$ as $E_{u}$, and the space-wise set $\{x|(x,t)\in (u>0)\cap C_{1}(Z)\}$ as $Q(t)$. Hence $|E_{u}|=\int^{1}_{0}|Q(t)|dt$. We need to prove 
\begin{equation}\label{equ simplified L1-est}
\int_{C_{\frac{1}{2}}(Z)}u_{+}dV_{E}ds\leq 
N|E_{u}|^{1+\frac{1}{2n+2}}
\end{equation}
To show (\ref{equ simplified L1-est}) is true, it suffices to show that for any $\epsilon$ small enough,  $u_{\epsilon}$ satisfies 
\begin{equation}
\int_{C_{\frac{1}{2}}(Z)}u_{\epsilon}dV_{E}ds\leq 
N|E_{u_{\epsilon}}|^{1+\frac{1}{2n+2}}.
\end{equation}
The advantage of $u_{\epsilon}$ is that it's supported in $Q(t)$, and $0\leq u_{\epsilon}\leq 1$. Then  integration by parts implies the energy estimates in Lemma \ref{lem: parabolic inequality holds in the right way} holds true over $Q(t)$. Let $\eta$ be the standard cut-off function in $C_{1}(Z)$ which vanishes near the parabolic boundary, H\"older's inequality and Lemma \ref{lem: parabolic inequality holds in the right way} imply
 \begin{equation}\label{equ 1 in Lem of L1 estimate}
 \int_{B}\eta u_{\epsilon}dV_{E}|_{t}\leq |Q(t)|^{\frac{1}{2}}(\int_{B}\eta^{2}u^{2}_{\epsilon}dV_{E})^{\frac{1}{2}}|_{t}\leq NE_{u_{\epsilon}}^{\frac{1}{2}}|Q(t)|^{\frac{1}{2}}.
 \end{equation}
 We also have the following bootstrapping estimate on the same term. 
 \begin{eqnarray}\label{equ the only place where Sobolev imbedding is used explicitly}& & \int_{B}\eta u_{\epsilon}dV_{E}
\leq  (\int_{B}|\eta u_{\epsilon}|^{\frac{2n}{2n-1}}dV_{E})^{\frac{2n-1}{2n}}|Q(t)|^{\frac{1}{2n}}\leq  N(\int_{B}|\nabla (\eta u_{\epsilon})|dV_{E}) |Q(t)|^{\frac{1}{2n}}\nonumber
 \\&\leq & N(\int_{B}|\nabla (\eta u_{\epsilon})|^{2}dV_{E})^{\frac{1}{2}}|Q(t)|^{\frac{1}{2n}+\frac{1}{2}}\ [\textrm{since}\ supp \nabla (\eta u_{\epsilon})\subset \{u>0\}\cap B].
 \end{eqnarray}
By (\ref{equ 1 in Lem of L1 estimate}), (\ref{equ the only place where Sobolev imbedding is used explicitly}), Lemma \ref{lem: parabolic inequality holds in the right way}, and Fubini-Theorem [with the help of (\ref{equ parabolic energy is bounded})], 
  \begin{eqnarray}& & \int_{-1}^{0}\int_{B}\eta u_{\epsilon}dV_{E}ds\nonumber= \int_{-1}^{0}(\int_{B}\eta u_{\epsilon}dV_{E})^{\frac{1}{n+1}}(\int_{B}\eta u_{\epsilon}dV_{E})^{\frac{n}{n+1}}ds
\\&\leq & N E_{u_{\epsilon}}^{\frac{1}{2n+2}}\int^{0}_{-1}|Q(t)|^{\frac{n+2}{2n+2}}(\int_{B}|\nabla (\eta u_{\epsilon})|^{2}dV_{E})^{\frac{n}{2n+2}}dt
 \\& \leq & N |E_{u_{\epsilon}}|^{\frac{1}{2n+2}}(\int^{0}_{-1}|Q(t)|dt)^{\frac{n+2}{2n+2}} (\int^{0}_{-1}\int_{B}|\nabla (\eta u_{\epsilon})|^{2}dV_{E}dt)^{\frac{n}{2n+2}} \nonumber
\leq N |E_{u_{\epsilon}}|^{1+\frac{1}{2n+2}}
 \end{eqnarray}
 Since $\eta\equiv 1$ over $C_{\frac{1}{2}}(Z)$, the proof is complete. As we've seen, \textbf{nothing in this proof involves more than Lemma \ref{lem: parabolic inequality holds in the right way} on the sub-solutions}.
\end{proof}

\begin{prop}\label{prop Moser iteration equivalent to growth lemma}
Suppose $u$ is a weak sub solution to (\ref{equ sub heat equation of g}) in a  cylinder $C_{r}(Y)$, $y\notin D$. Then 
\begin{equation}
u(Y)\leq \frac{N}{|C_{r}|}(\int_{C_{r}(Y)}u_{+}dV_{E}).
\end{equation}

\end{prop}
\begin{proof}{of Proposition \ref{prop Moser iteration equivalent to growth lemma}:} The proof is exactly as of Theorem 3.4 in \cite{FS}. The only thing worth mentioning is that we should deal with the singularity $D$. In \cite{FS}, they  consider the maximal point of $d^{\gamma}u$, where $\gamma =\frac{2n+2}{p}$ and $d$ is the parabolic distance to the the parabolic boundary of $C_{r}(Y)$. However, when singularity is present, $d^{\gamma}u$ might not attain maximum away from $D$. To overcome this, we  simply assume $u(Y)>0$, and use the fact that there exist an almost maximal point away from $D$. Namely, there exist $X_{0}=(x_{0},t_{0})$ such that $x_{0}\notin D$ and 
\begin{equation}\label{equ almost maximal point}
d^{\gamma}(X_{0})u(X_{0})\geq \frac{M}{2},\ M\triangleq \sup_{C_{r}}d^{\gamma}u 
\end{equation}
$(\textrm{we can assume}\ M>0\ \textrm{with out loss of generality})$. Then the rest of the proof is line by line as from line 13 to line -3 in Page 101 of \cite{FS}, except the $\mu_{1}$ in line 19 should correspond to $\beta_{1}=2^{-\gamma-2}$, because we have an additional $\frac{1}{2}$ in (\ref{equ almost maximal point}). \end{proof}

 \begin{rmk}  As mentioned in Remark 3.5 in \cite{FS}, this proof does NOT involve explicitly   the sub equation (\ref{equ sub heat equation of g}). Instead, it only requires the  Growth Lemma (\ref{prop growth thm 1}). Thus the conditions in (\ref{equ conditions of the metrics}) is not involved explicitly in  this proof. 
\end{rmk}

\begin{thm}(Slant Cylinder Theorem)\label{Thm Slant Cylinder}
Suppose $u$ is a weak sub solution to (\ref{equ sub heat equation of g}) in a slant cylinder $SC_{r}$. Suppose $u\leq 0$ in $B_{r}\times \{T_{0}\}$. Then 
\begin{equation}\label{equ in slant cylinder theorem}
u(Y)\leq (1-\lambda)\sup_{SC_{r}} u^{+}, 
\end{equation}
where $\lambda \in (0,1)$ depends on $n,\beta, \frac{r|y_{1}-y_{0}|}{T_{1}-T_{0}}\ (|l|), T_{1}-T_{0},$  and $K$. 
\end{thm}
\begin{proof}[Proof of Theorem \ref{Thm Slant Cylinder}:] \textbf{The first  paragraph in  the proof of Proposition \ref{prop growth thm 1} also applies here}.  By translation and rescaling (see Definition \ref{Def slant cylinder}), without changing the parabolic slope, we can transform  $SC_{r}$ to a slant cylinder $SC_{1}$ with $r=1,\ T_{0}=0,\ T_{1}=T,\ y_{0}=\{0\}$, and $ y_{1}=y$. We then pull back $u$ and \textbf{the matrix of the metric $g$ on $SC_{r}$} to "$u$" (by abuse of notation) and $\widehat{g}$ on $SC_{1}$. Thus, $u$ satisfies in $SC_{1}$ the following.
\begin{equation}\label{equ 1 subsolution in slant cylinder Thm}
\frac{\partial u}{\partial t}-\Delta_{\widehat{g}}u\leq 0\ \textrm{in the sense of Definition \ref{Def weak solution}, and}
\end{equation}
\begin{equation}\label{equ quasi-isometric condition in slant cylinders}
\frac{g_{Euc}}{K}\leq \widehat{g}\leq Kg_{Euc}\ \textrm{in}\ SC_{1}.
\end{equation}

It suffices to prove (\ref{equ in slant cylinder theorem}) for $u_{\epsilon}$.  By rescaling, we can assume $u \leq 1$ and $sup_{SC_{1}}u=1$. Then $0\leq u_{\epsilon} \leq 1-\epsilon$ and $sup_{SC_{1}}u_{\epsilon}\geq 1-3\epsilon$. It sufficies to  derive an estimate for for $v=-\log(1-u_{\epsilon})$
which is independent of  $\epsilon$. Since $u_{\epsilon}$ satisfies (\ref{equ 1 subsolution in slant cylinder Thm}), $v$ satisfies 
\begin{equation}
\frac{\partial v}{\partial t}-\Delta_{\widehat{g}}v\leq -|\nabla_{\widehat{g}}v|^2
\end{equation}
in the sense of Definition \ref{Def weak solution}. Let $\underline{\eta}$  be the standard cut-off function in the Euclidean unit ball $B(1)$ which only depends on $|x|^{2}$. By (proof of) Lemma \ref{lem: parabolic inequality holds in the right way} [\textbf{replace the $0$ on the right hand side of} (\ref{equ sub heat equation of g}) by $-|\nabla_{\widehat{g}}v|^{2}$], using $u_{\epsilon}\geq 0, u_{\epsilon}|_{t=0}=0$, by abuse of notation with Lemma \ref{lem: parabolic inequality holds in the right way},   we  consider  $\eta=\underline{\eta}[x-y(t)]$ and  obtain (similarly to (\ref{equ weak solution identity})) 
\begin{eqnarray}\label{equ 1 Thm slant cylinder}
& &\int_{\C^{n}} v\eta^2 dV_{g}|^{t_{2}}_{t_{1}}+\int^{t_{2}}_{t_{1}}\int_{\C^{n}} <\nabla_{g}v,\nabla_{g}\eta^2 >dV_{g}ds +\int^{t_{2}}_{t_{1}}\int_{\C^{n}} |\nabla_{g}u|^2 \eta^2 dV_{g}ds \nonumber
\\&\leq & \int^{t_{2}}_{t_{1}}\int_{\C^{n}} v\frac{\partial \eta^2}{\partial t} dV_{g}ds+K\int^{t_{2}}_{t_{1}}\int_{\C^{n}}v \eta^2 dV_{g}ds.
\end{eqnarray}

We first estimate the term $\int^{t_{2}}_{t_{1}}\int_{\C^{n}} v\frac{\partial \eta^2}{\partial t} dV_{g}ds$.  It's  the same as in \cite{FS}. We note that $|\frac{\partial \eta^2}{\partial t}|\leq |l||\nabla_{E}\eta^2|$ (Definition \ref{Def slant cylinder}). Then 
\begin{equation}
|\int^{t_{2}}_{t_{1}}\int_{\C^{n}} v\frac{\partial \eta^2}{\partial t} dV_{g}ds|\leq |l|K^{2n}\int^{t_{2}}_{t_{1}}\int_{\C^{n}} v|\nabla_{E}\eta^2|dV_{E}ds.
\end{equation}
Using  line 14 to line 23 in page 103 of \cite{FS}, we obtain
\begin{equation}
\int_{\C^{n}} v|\nabla_{E}\eta^2|dV_{E}\leq N\int_{\C^{n}}( |v| +|\nabla_{E}v|)\eta^2 dV_{E}.
\end{equation}
Then  Cauchy-Schwartz inequality and the quasi-isometric condition (\ref{equ quasi-isometric condition in slant cylinders}) imply
\begin{eqnarray}\label{equ 2 Thm slant cylinder}
& &|\int^{t_{2}}_{t_{1}}\int_{\C^{n}} v\frac{\partial \eta^2}{\partial t} dV_{g}ds|\nonumber
\\&\leq & \frac{1}{100}\int^{t_{2}}_{t_{1}}\int_{\C^{n}} |\nabla_{g}v|^{2}\eta^2 dV_{g}ds+N+N\int^{t_{2}}_{t_{1}}\int_{\C^{n}} v\eta^2 dV_{g}ds.
\end{eqnarray}
By the same reason we have 
\begin{equation}\label{equ 3 Thm slant cylinder}
|\int^{t_{2}}_{t_{1}}\int_{\C^{n}} <\nabla_{g}v,\nabla_{g}\eta^2 >dV_{g}ds|\leq  \frac{1}{100}\int^{t_{2}}_{t_{1}}\int_{\C^{n}} |\nabla_{g}v|^{2}\eta^2 dV_{g}ds+N.
\end{equation}
Then (\ref{equ 1 Thm slant cylinder}), (\ref{equ 2 Thm slant cylinder}), and (\ref{equ 3 Thm slant cylinder}) imply
\begin{equation}\label{equ 4 Thm slant cylinder}
\int_{\C^{n}} v\eta^2 dV_{g}|^{t_{2}}_{t_{1}}
\leq  N+N\int^{t_{2}}_{t_{1}}\int_{\C^{n}}v \eta^2 dV_{g}ds.
\end{equation}
Denote $\int_{\C^{n}} v\eta^2 dV_{g}|_{t}=I(t)$, since $I(0)=0$, (\ref{equ 4 Thm slant cylinder}) implies 
$I(t)$ satisfy the assumption in Lemma \ref{lem bounding on I from the integral inequality}. Hence  Lemma \ref{lem bounding on I from the integral inequality} implies  $I(t)\leq N$ for all $t\in [0,T]$. Then  Proposition \ref{prop Moser iteration equivalent to growth lemma} implies $v(Y)\leq N$. Hence  for some $\lambda$ (as in Theorem \ref{Thm Slant Cylinder}) which is  independent of $\epsilon$, $u_{\epsilon}(Y)\leq 1-2\lambda\leq (1-\lambda)\sup_{SC_{1}}u_{\epsilon}$ when $\epsilon$ is small enough. Let $\epsilon \rightarrow 0$,  the proof of (\ref{equ in slant cylinder theorem}) is complete.  Again, \textbf{nothing in this proof involves more than the energy estimates of the sub-solutions}.
\end{proof}

\begin{lem}\label{lem bounding on I from the integral inequality}  Suppose $I(t),\ t\in [T_{0},T_{1}]$ is an everywhere defined $L^{\infty}$ function.  
Suppose $I(t)\geq 0$ for all $t$, $I(T_{0})=0$,  and  \begin{equation}
I(t)\leq I(t_{1})+N_{1}\int^{t_{2}}_{t_{1}}I(s)ds+N_{2},\ \textrm{for all}\ t_{1},\ t_{2}\ \textrm{and}\ t\in [t_{1},t_{2}].
\end{equation}
Then there exists $N$ depending on $N_{1}$, $N_{2}$, and $T_{1}-T_{0}$ such that $I(t)\leq N$.

\end{lem}
\begin{proof} Choose $a$ such that  $a\leq \frac{1}{100N_{1}}$ and $\frac{T_{1}-T_{0}}{a}=k_{0}$ is an integer. Then for $k\leq k_{0}-1$, we deduce $\max_{ka\leq t\leq (k+1)a}I(t)\leq \frac{1}{2}\max_{ka\leq t\leq (k+1)a}I(t)+N_{2}+I(ka),\ \textrm{then}$
\begin{equation}
\max_{ka\leq t\leq (k+1)a}I(t)\leq 2N_{2}+2I(ka).
\end{equation}
Since $I(T_{0})=0$,  the proof is complete by induction.
\end{proof}

The short proof of Theorem 4.2 in \cite{FS} (only involving Theorem \ref{Thm Slant Cylinder}) gives
\begin{prop}\label{prop growth thm 2} Suppose $u$ is a weak sub solution to (\ref{equ sub heat equation of g}) in a  cylinder $C_{r}(Y)$. Suppose 
$u\leq 0$ on    $B_{\rho}(z)\times \{\tau\}$, where   $s-r^{2}\leq \tau<s-\frac{r^{2}}{4}-\rho^2$. Then 
\begin{equation}
\sup_{C_{\frac{r}{2}}}u\leq (1-\lambda)\sup_{C_{r}}u,\ \textrm{where}\ \lambda \in (0,1)\ \textrm{depends on}\ n,\ \beta,\ \frac{\rho}{r},\ K.
\end{equation}
\end{prop}

\begin{thm}(Main Growth Theorem)\label{thm growth theorem} Suppose $u$ is a weak sub solution to (\ref{equ sub heat equation of g}) in a  cylinder $C_{2r}(Y)$. Suppose 
\begin{equation}
\frac{|\{u>0\}\cap C_{r}(y,s-3r^{2})|}{|C_{r}(y,s-3r^{2})|}\leq \frac{1}{2}.
\end{equation}
Then 
\begin{equation}
\sup_{C_{r}}u\leq (1-\lambda)\sup_{C_{2r}}u^{+}, \textrm{where}\ \lambda \in (0,1)\ \textrm{depends on}\
n,\beta,K.
\end{equation}
\end{thm}
\begin{proof}[Proof of Theorem \ref{thm growth theorem}:] Instead of  directly quoting  the  work in \cite{FS}, we would like to make the crucial point: 

 \textit{Except measure theory which does not involve the sub equation (\ref{equ sub heat equation of g}), the proof of Theorem 5.3 in \cite{FS} only depends on the fact that Proposition \ref{prop growth thm 1} (Theorem 3.3 in \cite{FS}) and \ref{prop growth thm 2} (Theorem 4.2 in \cite{FS}) hold true for any sub solution (with suitable conditions on initial value or level sets) in any scale.}

Actually both propositions are applied in case (a) in page 109 of \cite{FS}. 

Thus, using Proposition \ref{prop growth thm 1} (in the position of Theorem 3.3 in \cite{FS}) and \ref{prop growth thm 2} (in the position of Theorem 4.2 in \cite{FS}), the proof of Theorem 5.3 in \cite{FS} goes through for Theorem \ref{thm growth theorem}.
\end{proof}

 Yuanqi Wang, Department of Mathematics, Stony Brook University, Stony Brook, NY, USA;\ \ ywang@scgp.stonybrook.edu.
\end{document}